\documentclass[10pt]{amsart}
\usepackage[]{amsmath}
\usepackage{amssymb}
\usepackage{float}
\usepackage{titlesec}
\usepackage[]{inputenc}
\titlelabel{\llap{\thetitle\quad}}
\usepackage[]{color}
\renewcommand{\epsilon}{\varepsilon}
\usepackage{url}
\usepackage{graphicx}
\usepackage[]{color}
\newtheorem*{theorem*}{Theorem}
\theoremstyle{plain}
\newtheorem{theorem}{Theorem}[section]
\newtheorem{prop}[theorem]{Proposition}
\newtheorem{lemma}[theorem]{Lemma}

\usepackage{tikz}
\theoremstyle{definition}
\newtheorem{remark}[theorem]{Remark}
\numberwithin{equation}{section}

\newcommand{\torus}{\mathbb{T}^n }

\newcommand{\nega}{ \psi(t)\hat{a}(t/\loga)}
\newcommand{\R}{\mathbb{R}}

\newcommand{\re}{ {\rm{Re }\, }}
\newcommand{\im}{ {\rm{Im }\, }}
\newcommand{\loga}{ {\varepsilon\log\lambda}}

\newcommand{\Z}{\mathbb{Z}}

\newcommand{\srho}{\tilde{\rho}(\lambda t)}

\newcommand{\local}{\mathfrak{S}_{loc}}

\newcommand{\coe}{{i(\lambda+i\mu)}}
\newcommand{\dis}{d_g(x,y)}

\titleformat*{\section}{\bf}
\titlespacing{\section}{0em}{2ex}{2ex}


\author{Peng Shao }
\email{pshao@math.jhu.edu}
\address{Department of Mathematics, Johns Hopkins University, Baltimore, MD 21218, U.S.A.}

\author{Xiaohua Yao}
\email{yaoxiaohua@mail.ccnu.edu.cn}
\address{Department of Mathematics, Huazhong Normal University, Wuhan 430079, PR China}

\title[General sobolev resolvent $L^p$ estimates]
{ Uniform Sobolev Resolvent Estimates for the Laplace-Beltrami operator on Compact Manifolds}

\thanks{The second author is supported by Projects for New Century Excellent Talents in
University (NCET-10-0431) and the Special Fund for Basic Scientific Research of Central Colleges (No. CCNU12C01001)}

\begin{document}
\maketitle

\begin{abstract}
	In this paper we continue the study on the resolvent estimates of the
	Laplace-Beltrami operator $\Delta_g$ on a compact manifolds $M$ with dimension
	$n\geq3$. On
	the Sobolev line $1/p-1/q=2/n$ we can prove that the resolvent
	$(\Delta_g+\zeta)^{-1}$ is uniformly bounded from
	$L^p$ to $L^q$ when $(p,q)$ are within the range: $p\leq2(n+1)/(n+3)$ and
	$q\geq2(n+1)/(n-1)$ and $\zeta$ is outside a parabola opening to the right and a
	small disk centered at the origin.
	This naturally generalizes the previous results in
	\cite{Kenig} and \cite{bssy} which addressed only the special case when
	$p=2n/(n+2), q=2n/(n-2)$. 
	Using the shrinking spectral
	estimates between $L^p$ and $L^q$ we also show that when $(p,q)$ are  
	within the interior of the 
	range mentioned above, one can obtain a logarithmic improvement over the
	parabolic region for resolvent estimates on manifolds equipped with Riemannian
	metric of non-positive sectional curvature, and a power improvement depending on
	the exponent $(p,q)$ for flat torus.
	The latter therefore partially improves Shen's work in \cite{Shen} on the
	$L^p\to L^2$ 
	uniform resolvent estimates on the torus.
	Similar to the case as proved in \cite{bssy} when $(p,q)=(2n/(n+2),2n/(n-2))$, the
	parabolic region is also optimal over the round sphere $S^n$ when $(p,q)$  are now in the
	range.  
	However, we may ask if the  range is 
	 sharp in the sense that it is the only
	possible range on the Sobolev line for which a compact manifold can have uniform
	resolvent estimate for $\zeta$ being ouside a parabola. 
	
\end{abstract}

\section{Introduction}
Recall that in \cite{Kenig} (see also \cite{bssy}) Dos Santos Ferreira, Kenig and Salo
proved the following result concerning the resolvent
estimates on a compact boundaryless Riemannian manifold:
\begin{theorem}
	Let $(M,g)$ be a compact Riemannian manifold of dimension $n\geq 3$, and let
	$\lambda,|\mu|\geq1$. Then there exists a uniform constnat $C>0$ such that for
	all $f\in C^{\infty}(M)$ we have the following resolvent estimate
	\begin{equation}
		||f||_{L^{\frac{2n}{n-2}}(M)}\leq
		C||(\Delta_g+(\lambda+i\mu)^2)f||_{L^{\frac{2n}{n+2}}(M)}.
		\label{Kenig_estimate}
	\end{equation}
	\label{Kenig.theorem}
\end{theorem}
Notice that if we write $\zeta=(\lambda+i\mu)^2$ then it is outside a small disk and a parabola opening to the right as
in the following figure:
\begin{center}
	\begin{tikzpicture}
\draw[domain=0.5:2.1, rotate=270,yshift=0.2cm] plot (\x,{0.3*(\x)^3});
\draw[domain=-2.1:-0.5, rotate=270,yshift=0.2cm] plot (\x,{0.3*(-\x)^3});
\draw (0,0) circle (0.55cm);
\draw[->](-2,0)--(5,0) node[above]{$\re\zeta$};
\draw[->](0,-2)--(0,2) node[right]{$\im\zeta$};
	\end{tikzpicture}
\end{center}

Dos Santos Ferreira, Kenig and Salo \cite{Kenig} used explicit Hadamard parametrix construction to obtain the
estimates above which is based on a classical representation of such parametrix in terms
of Bessel
functions. See also \cite{KRS}. Shortly after, Bougain, Sogge and us in \cite{bssy}
showed that estimate \eqref{Kenig_estimate} is sharp on round sphere. They also 
used half-wave operator $e^{it\sqrt{-\Delta_g}}$ and $\cos t\sqrt{-\Delta_g}$
to prove the equivalence between any possible improvement over the parabola and shrinking spectral
projection estimate of $\sqrt{-\Delta_g}$, and obtained some improvements on the torus and non-positive
curvature manifolds. In particular, using this technique they could obtain a 
shorter proof to Theorem \ref{Kenig.theorem}.

The specific $(p,q)$ pair appearing in \eqref{Kenig_estimate} is at the intersection of
the line of duality $1/p+1/q=1$ and the Sobolev line $1/p-1/q=2/n$. Interestingly in the current
paper we show that the line of duality does not play a significant role here, and
the parabolic boundary of the region is essentially the result of the Sobolev line. More
explicitly, we can prove that:
\begin{theorem}
	Let $M$ be a compact Riemannian manifold of dimension $n\geq 3$. Then, if
	$1/p-1/q=2/n$, we have the following uniform resolvent estimates
	\begin{equation}
		||f||_{L^{q}(M)}\leq C||(\Delta_g+(\lambda+i\mu)^2)f||_{L^p(M)}
		\label{main_estimate}
	\end{equation}
	if ${p\leq 2(n+1)}/{(n+3)}$ and $q\geq2(n+1)/(n-1)$,
	and $\lambda,|\mu|\geq1$. In particular, the constant $C$ does not depend on $\lambda,\mu$.
	\label{main_theorem}
\end{theorem}
We follow our original way in \cite{bssy} to prove this theorem by splitting the
resolvent into short-time local part and long-time non-local remainder. The way we handle
the local part is similar to, and motivated by the work in \cite{Kenig} and
\cite{Sogge1} through using the 
Carleson-Sj\"{o}lin condition of an
oscillatory term which we did not use in \cite{bssy} since we concerned only the
$L^{\infty}$ norm of the kernel at that time. The main difference between our work and
\cite{Kenig} is the way we handle the remainder term,  whose $L^p$ norm on the Sobolev
line we are able to control by an argument using Sogge's 
spectral projection estimates.

The paper is organized as the following. As usual we interpret the resolvent
$(\Delta_g+(\lambda+i\mu)^2)^{-1}$ as a multiplier 
$-(\tau^2-(\lambda+i\mu)^2)^{-1}(P)$, in which and following on, $P$ denotes $\sqrt{-\Delta_g}$. We
then calculate the Fourier transform of this multiplier function, and use the half-wave
operator and Fourier inverse transform formula to write the resolvent as
\begin{equation}\begin{split}
	&\frac{\rm{sgn}\mu}{2i(\lambda+i\mu)}\int_{-\infty}^{\infty}e^{i(\rm{sgn}\mu)\lambda|t|}e^{-|\mu t|}e^{itP}dt
	=\frac{\rm{sgn}\mu}{\coe}\int_0^{\infty}e^{i(\rm{sgn}\mu)\lambda t}e^{-|\mu| t}\cos tPdt.
\end{split}
	\label{resolvent}
\end{equation}
Consequently we use a smooth function $\rho(t)$ supported near $t=0$ to split the resolvent in
\eqref{resolvent} into local and non-local parts:
\begin{equation}
	\local(P)
	=\frac{\rm{sgn}\mu}{i(\lambda+i\mu)}\int_0^{\infty}\rho(t)e^{i\lambda t}e^{-|\mu| t}\cos{tP}dt
\label{local_operator_definition}
\end{equation}
and
\begin{equation}
	r_{\lambda,\mu}(P)
	=\frac{\rm{sgn}\mu}{i(\lambda+i\mu)}\int_0^{\infty}(1-\rho(t))e^{i\lambda
	t}e^{-|\mu|
	t}\cos{tP}dt.
\label{non_local_operator_definition}
\end{equation}

In Section 3 we study the non-local operator
again in similar way as we did in \cite{bssy}, by breaking the spectrum of $P$ into
unit-length clusters and then estimate the $(p,q)$ norm of the multiplier $r_{\lambda,\mu}$ on
each piece with the help of Sogge's spectral estimates in, say \cite{Soggebook}. Our main
tool is Lemma \ref{lemma2.3} which is a variant of Lemma 2.3 in \cite{bssy}. The
difference between these two lemmas is that instead of the standard
$TT^*$ argument we now consider the composition of spectral projections from $L^p$ to
$L^2$ and from $L^2$ to $L^q$ in an asymmetric manner, which happens to behave well on the Sobolev line. By
combining the results for local and non-local operator together, the proof to Theorem
\ref{main_theorem} is therefore completed.

Similar to the importance of Lemma 2.3 in \cite{bssy}, our Lemma \ref{lemma2.3} can
immediately derive the following relation between shrinking spectral projection
$(p,q)$ estimates and the improved uniform resolvent estimates. Notice that unlike the
case on the line of duality, we are not able to prove the exact equivalence between them:
\begin{theorem}
	Let $M$ be a compact Riemannian manifold of dimension $n\geq3$.
	Suppose that for $2n(n+1)/(n^2+3n+4)\leq p\leq 2(n+1)/(n+3)$ we have a function
	$0<\varepsilon _p(\lambda)\leq1$ decreasing monotonically to $0$ as $\lambda\to\infty$
	and $\varepsilon_p(2\lambda)\geq\varepsilon_p(\lambda)/2$, for $\lambda$
	sufficiently large. Then if we have 
	\begin{equation}
		||\sum_{|\lambda-\lambda_j|\leq\varepsilon
		_p(\lambda)}E_jf||_{L^{p'}(M)}\leq
		C\varepsilon_p(\lambda)\lambda^{2\delta(p)}||f||_{L^p(M)},\lambda\gg1,
		\label{1.13}
	\end{equation}
	 we also have the following resolvent estimates for $1/p-1/q=2/n,
	p\leq2(n+1)/(n+3), q\geq2(n+1)/(n-1)$:
	\begin{equation}
		||f||_{L^q(M)}\leq C||(\Delta_g+(\lambda+i\mu)^2)f||_{L^p(M)},
		|\mu|\geq\max\left\{ \varepsilon_p(\lambda),\varepsilon _{q'}(\lambda)
		\right\}, \lambda\gg1.
		\label{1.14}
	\end{equation}
	\label{equivalence}
\end{theorem}
With this theorem, we show in section 4 that the uniform resolvent estimates in
Theorem \ref{main_theorem} can be improved if the
manifold $M$ is equipped with a Riemannian metric with non-positive sectional curvature.
More precisely we can prove the following theorem:
\begin{theorem}
	If $M$ is a boundaryless Riemannian compact manifold with dimension $\geq3$ and of
	non-positive sectional curvature, then for $1/p-1/q=2/n, p<2(n+1)/(n+3),
	q>2(n+1)/(n-1)$ we have the following uniform resolvent estimates
	\begin{equation}
		||f||_{L^q(M)}\leq C ||(\Delta_g+(\lambda+i\mu)^{2})f||_{L^p(M)}
	\end{equation}
	if $\lambda\gg1$ and $|\mu|>(\log(\lambda))^{-1}$.
	\label{neg.result}
\end{theorem}

The above theorem is an example in which one can get the same regional improvement for many
$(p,q)$ pairs on the Sobolev line, which is due to the slow growth of $\log\lambda$
compared to any power of $\lambda$. In general the improvements may  unsurprisingly
depend on the concrete value of $(p,q)$ as we have seen in Theorem \ref{equivalence}. 
In section 5 we prove the following theorem about the
improved resolvent estimates on Torus $\torus$ for $n\geq 3$ which serves as such an exmple:
\begin{theorem}
	Let $\torus$ denote the flat torus with $n\geq 3$. Then for a $(p,q)$ pair satisfying
	$1/p-1/q=2/n$ and $p\leq2(n+1)/(n+3), q\geq2(n+1)/(n-1)$ there exists a function in
	$p$, which we denote by
	$\epsilon_n(p)$, such that when $1/p$ is ranging from $(n+3)/2(n+1)$ to $(n+2)/2n$
	($\overline{AF}$ in figure \ref{figure}) it increases from 
	from $0$ to $1/(n+1)$, and symmetrically decreases from $1/(n+1)$ to $0$ when
	$(n+2)/2n\leq 1/p\leq (n^2+3n+4)/2n(n+1)$ ($\overline{FA'}$ in figure \ref{figure}),
	and we have the following improved resolvent
	estimates
	\begin{equation}
		||f||_{L^q(\torus)}\leq
		C||(\Delta_{\torus}+(\lambda+i\mu)^2)f||_{L^p(\torus)},
		\quad\lambda>1, |\mu|\geq\lambda^{-\varepsilon (p)}.
		\label{torus.estimate}
	\end{equation}
	The exact form for $\varepsilon_n(p)$ is given in \eqref{epsilon.p} when
	$(1/p,1/q)$ is below the line of duality.
	\label{torus.theorem}
\end{theorem}
Recall that in \cite{Shen} Shen proved the following uniform resolvent estimates:
\begin{equation}
	||f||_{L^{q}(\torus)}\leq
	C||(\Delta_{\torus}+(\lambda+i\mu)^2)f||_{L^2(\torus)},\quad \lambda,|\mu|\geq1
	\label{shen}
\end{equation}
for $2\leq q<2(n-2)/(n-4)$ when $n\geq 4$ and $2\leq q\leq\infty$ when $n=3$. If we use
H\"{o}lder inequality based on the fact that $\torus$ is compact and $p<2$ in
\eqref{torus.estimate}, we can obtain similar $L^2\to L^q$ type resolvent estimates but
with a much smaller $q$-range compared with \eqref{shen}. Our Theorem \ref{torus.theorem} on the
other hand improves
Shen's estimates in the aspect of allowing a smaller $|\mu|$ comparable to certain negative power of
$\lambda$ for part of his exponent range.
	
The following figure \ref{figure} can be used by the interested readers to understand the range of
the $(p,q)$ pair. $\overline{AA'}$ will be from
the non-local operator, which happens to be the global range mentioned in Theorem
\ref{main_theorem}. The Carleson-Sj\"{o}lin argument used for local operator will give us
the $\overline{DO}$ and Young's inequality can give us 
segment $\overline{EO}$, therefore we can interpolate to obtain the
segment $\overline{CC'}$ with both end points removed as the range for the local
part. Since in general $\overline{AA'}\subset\overline{CC'}$, the resolvent estimate
range for a compact manifold is therefore, as far as we can prove, is
constrained in $\overline{AA'}$. 
Notice that point $F=(\frac{n+2}{2n},\frac{n-2}{2n})$ is the $(1/p,1/q)$ pair considered in
\cite{Kenig} and \cite{bssy}.
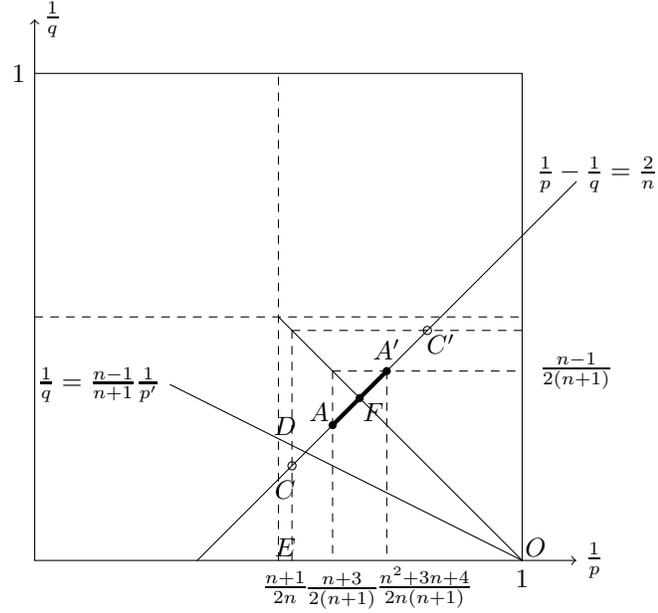
\begin{figure}[h]
	\hspace{1cm}\begin{tikzpicture}[scale=3.6,cap=round, dot/.style={circle,fill=black,minimum size=4pt,inner sep=0pt,
            outer sep=-1pt}]
  \tikzstyle{axes}=[]
  \tikzstyle{important line}=[very thick]
  \tikzstyle{information text}=[rounded corners,fill=red!10,inner sep=1ex
  ]

  \draw[style=help lines,step=0.1cm] ;
  \draw[->] (0,0) -- (2,0) node[right]{$\frac{1}{p}$};
  \draw[->] (0,0) -- (0,2) node[right]{$\frac{1}{q}$};	
  \draw[-] (1.8,1.8) -- (0,1.8) node[left]{$1$};
  \draw[-] (1.8,1.8) -- (1.8,0) node[below]{$1$};
  \draw[dashed] (0,0.9) -- (1.8,0.9);
  \draw[dashed] (0.9,0) -- (0.9,1.8);
  \draw[dashed] (1.1, 0.7) -- (1.1,0) ;
   \draw[dashed] (1.1,0.7) -- (1.8,0.7);
  \draw[ultra thick,-] (1.1,0.5) -- (1.3,0.7);
  \draw[-] (0.6,0) -- (2,1.4) node[xshift=-0.65cm,yshift=0.1cm,right]{$\frac{1}{p}-\frac{1}{q}=\frac{2}{n}$};
\fill(1.1,0.5) circle (0.015cm);
\fill(1.3,0.7) circle (0.015cm);
\draw(1.45,0.85) circle (0.015cm);
\draw[dashed](1.3,0.7)--(1.3,0);
\fill(1.2,0.6) circle (0.015cm);
\node at (1.25,0.55) {$F$};
\node at (1.05,0.55) {$A$};
\node at (1.3,0.78) {$A'$};
\draw[-] (1.8,0) -- (0.5,0.65)
node[left]{$\frac{1}{q}=\frac{n-1}{n+1}\frac{1}{p'}$};
\node at (0.925,0.495) {$D$};
\draw (0.95, 0.35) circle (0.015cm);
\node at (0.925,0.260) {$C$};
\draw[dashed](0.95,0.85)--(0.95,0);
\node at (0.925, 0.05) {$E$};
\node at (0.925, -0.1) {$\frac{n+1}{2n}$};
\node at (1.135, -0.11) {$\frac{n+3}{2(n+1)}$};
\node at (1.44, -0.1){$\frac{n^2+3n+4}{2n(n+1)}$};
\draw[dashed](0.95,0.85)--(1.8,0.85);
\node at (1.5,0.8) {$C'$};
\draw[-](1.8,0)--(0.9,0.9);
\node at (2,0.7) {$\frac{n-1}{2(n+1)}$};
\node at (1.85,0.05) {$O$};
\end{tikzpicture}
\caption{Admissible pairs}
\label{figure}
\end{figure}
\begin{remark}
	In \cite{bssy} we showed that the parabolic region is
	sharp for round spheres with dimension $\geq3$ for pair
	$(\frac{2n}{n+2},\frac{2n}{n-2})$. A simple duality argument and interpolation
	will show that this region is also sharp for our Theorem \ref{main_theorem}.
	We can somehow ask a question concerning the sharpness in a different manner: \emph{Is
	the range $\overline{AA'}$ sharp?}
	More precisely, is it possible to find a larger range than
	$\overline{AA'}$ on the Sobolev line such that for a general compact manifold $M$ we have
	uniform resolvent estimates as in Theorem \ref{main_theorem}.  See Remark
	\ref{riesz.means} for more discussion on this problem.

\end{remark}
Throughout
this paper $\delta (p)=n|\frac{1}{p}-\frac{1}{2}|-\frac{1}{2}$ for $1\leq p\leq
\infty$ and unless specified otherwise we generally assume $1\leq p\leq2\leq q\leq\infty$.

\noindent{\bf Acknowledge:} We would like to thank C. D. Sogge for many helpful and
enlightening discussions during the completion of this paper.
\section{Local Operator}
Our main theorem in this section is
\begin{theorem}
	The local operator $\local(P)$ is uniformly bounded for $\lambda>1$ and
	$\mu\neq0$ from $L^p(M)$ to
	$L^q(M)$ if $1/p-1/q=2/n$ and $p<2n/(n+1), q>2n/(n-1)$, or more straightforwardly
	when $(1/p,1/q)$ is on the
	segment $\overline{CC'}$ in figure \ref{figure} with both end points removed.
	\label{local_theorem}
\end{theorem}
For simplicity we only prove the case when $\mu>1$, the other case is symmetric.
We first split the local resolvent into 
\begin{equation*}
	\local(x,y)=\sum_{j=0}^{\infty}S_j(x,y)
\end{equation*}
in which
\begin{equation}
	S_jf=\frac{1}{i(\lambda+i\mu)}\int_{0}^{\infty}\beta(\lambda2^{-j}t)\rho(t)e^{i\lambda
	t-\mu t}\cos{tP}dt,\quad j\geq1
	\label{S_j}
\end{equation}
and 
\begin{equation}
	S_0f=\frac{1}{i(\lambda+i\mu)}\int_0^{\infty}(1-\sum_{j=0}^{\infty}\beta(\lambda2^{-j}t))\rho(t)e^{i\lambda
	t-\mu t}\cos{tP}dt.
	\label{S_0}
\end{equation}
Here the function $\beta\in C_0^{\infty}(\R^1)$ satisfies the following properties
\begin{equation}
	\beta(t)=0, t\notin[1/2,2], |\beta(t)|\leq1, {\rm{\, and\, }}
	\sum_{-\infty}^{\infty}\beta(2^{-j}t)=1
	\label{littlewood.paley.function.property}
\end{equation}
Roughly speaking $S_0$ is the worse part of the local operator as its time support is
close to the singular point $t=0$. So instead of the oscillatory integral technique we are
going to use for $S_j,j\geq1$, we take advantage of the $O(\lambda^{-1})$ smallness of the
time support to directly estimate its kernel. Here we use a slightly simpler way to
achieve this compared with the one used in
\cite{bssy}. In fact, we are going to prove that
\begin{lemma}
	The multiplier $S_0(\tau)$ defined as
	\begin{equation}
		S_0(\tau)=\frac{1}{\coe}\int_0^{\infty}\srho\rho(t) e^{i\lambda t}e^{-\mu t}\cos
		t\tau dt
		\label{multiplier}
	\end{equation}
	is a $-2$ order symbol
	function with symbol norm independent from $\lambda$ or $\mu$.
	\label{-2symbol}
\end{lemma}
\begin{proof} Due to the small $t$ support in the integrand we know that when
	$|\tau|<1$ the integral, and similarly its $\tau$ derivatives in \eqref{multiplier}
	are uniformly bounded. Therefore we need only to prove that
	\begin{equation}
		|\frac{d^j}{d\tau^j}S_0(\tau)|\leq C_j
		\tau^{-2-j},|\tau|\geq1.\label{symbol.estimate}
	\end{equation}
	Let us prove first the case $j=0$ which may help the readers to understand how
	to handle the general case. Due to the fact that 
	\begin{equation}
		\cos t\tau=\frac{1}{\tau}\frac{d}{dt}\sin t\tau,\quad \sin
		t\tau=-\frac{1}{\tau}\frac{d}{dt}\cos t\tau\label{cos}
	\end{equation}
	we can do integration by parts twice and end up with some integrals
	boundary terms. Combining the fact that the integrand has a small
	$t$ support $t\leq4\lambda^{-1}$, $e^{-\mu t}\mu $ is integrable uniformly in 
	$\mu$ and $|\lambda+i\mu|\geq\lambda$ or $\mu$ we immediately see that both the
	boundary terms and the integrals are 
	uniformly bounded. This proves \eqref{symbol.estimate} when $j=0$.

        Now after taking $j$ times $\tau$ derivatives we  have
	\begin{equation*}
		\frac{d^j}{d\tau^j}S_0=\frac{\pm1}{\coe}\int\srho\rho(t) e^{i\lambda
		t}e^{-\mu t}t^j\cos t\tau dt\quad 
	\end{equation*}
	when $j$ is even, and with $\cos t\tau$ being replaced by $\sin t\tau$ when
	$j$ is odd. Then similar to the $j=0$ case take integration by
	parts in $t$ for $j+2$ times. Thanks to the presence of $t^j$ no matter $j$ is
	even or odd the boundary terms would be non-vanishing only at the final step,
	which can be estimated similarly as in the case $j=0$. So for simplicity we assume
	$j$ is even and ignore the boundary terms.

	Now by Leibniz's formula we have, for $\alpha,\beta,\gamma\geq0$,
	\begin{equation*}
		\begin{split}
			\frac{d^j}{d\tau^j}S_0&=\sum_{\alpha+\beta+\gamma=j+2}\frac{C_{\alpha\beta\gamma}}{\tau^{j+2}\coe}\int_0^{\infty}\frac{d^{\alpha}}{dt^{\alpha}}(
			e^{-\mu t} )\frac{d^{\beta}}{dt^{\beta}}(
			t^j )\frac{d^{\gamma}}{dt^{\gamma}}(\srho\rho(t)e^{i\lambda
			t})\cos t\tau dt\\
			&=\frac{1}{\tau^{j+2}\coe}(\sum_{\alpha+\beta=j+2}+\sum_{\alpha+\beta=j+1}+\sum_{\alpha+\beta\leq
			j})\\
			&=\frac{1}{\tau^{j+2}\coe}(I+II+III).
		\end{split}
	\end{equation*}
	Notice that in $I$ we have $\gamma=0$, in $II$ we have $\gamma=1$ and in
	$III$ we have $\gamma\geq2$. A simple check will show that terms in $I$ are those
	containing $\mu^{k+2}t^k,k\geq0$, therefore can be estimated using variable scaling
	$t\to\mu^{-1}t$ and the fact that $|\lambda+i\mu|>\mu$. Similarly the terms in
	$II$ are those containing $\mu^{k+1}t^k,k\geq0$ and can be handled using the same
	scaling and the fact $|\lambda+i\mu|>\lambda$. Finally, it is easy to check
	\begin{equation*}
		III=\sum_{\alpha+\beta\leq
		j}\frac{C_{\alpha\beta}}{\tau^{j+2}\coe}\int_0^{\infty}e^{-\mu
		t}(\mu
		t)^{\alpha}t^{j-\beta-\alpha}\left(\frac{d}{dt}\right)^{j-\beta-\alpha+2}(\srho\rho(t)e^{i\lambda
		t})\cos t \tau dt.
	\end{equation*}
	Using the facts that the factors $e^{-\mu t}(\mu t)^{\alpha}$ are uniformly
	bounded, on the support of the integrands we have $t<\lambda^{-1}$ and
	$|\lambda+i\mu|>\lambda$ the proof is therefore complete.
\end{proof}

With the aid of this lemma, we know that $S_0(P)$, defined in the sense of spectral theory
by $S_0(P)f=\sum_{j=0}^{\infty}S_0(\lambda_j)E_j f, f\in C^{\infty}(M)$ in which $E_j$ is
the $j$-th eigenspace projection associated with eigenvalue $\lambda_j$ of $P$, is therefore a $-2$ order
pseudodifferential operator (see for example \cite{Soggebook}, Theorem 4.3.1),
and in particular with symbol norms uniformly bounded from $\lambda$ or
$\mu$. This then leads us to the following kernel estimate (see for example Proposition 1
on the page 241 of
\cite{Stein}) if we recall that $n\geq3$:
\begin{equation*}
	|S_0(x,y)|\leq C\dis^{2-n}.
\end{equation*}
By the Hardy-Littlewood-Sobolev inequality, $S_0$ is a $L^{p}\to L^{q}$ bounded operator on
Sobolev line $\left\{ (1/p,1/q):1/p-1/q=2/n \right\}$.

To deal with $S_j,j\geq1$, we need the following version of Proposition 2.4 in
\cite{bssy}
\begin{prop}
	Let $n\geq2$ and assume that $a\in C^{\infty}(\R_+)$ satisfies the Mihlin-type
	condition that for each $j=0,1,\dots,2$ 
	\begin{equation}
		|\frac{d^j}{ds^j}a(s)|\leq A_js^{-j}, s>0
		\label{mihlin}
	\end{equation}
	Then there are constants $B,B_j$, which depend only on the size of finitely many of
	the constants $A_j$ so that for every $\omega\in\mathbb{C}$ such that for
	$\im\omega\neq0$ and $1/4<|x|<4$ we 
	have
	\begin{equation}\begin{split}
		\int_{\R^n}\frac{a(|\xi|)e^{ix\cdot\xi}}{|\xi|-\omega}d\xi
		=|x|^{1-n}a_{1,\omega}(x)+\sum_{\pm}e^{\pm i(\re\omega)|x|}
		|\re\omega|^{\frac{n-1}{2}}|x|^{-\frac{n-1}{2}}a_{2,\omega}^{\pm}(|x|),
		\end{split}\label{onedajifen}
	\end{equation}
	where $|a_{1,\omega}(x)|=B$ is a bounded smooth function,
	$|\frac{d^j}{ds^j}a_{2,\omega}^{\pm}(s)|\leq
	B_js^{-j}$. Therefore $a_{2,\omega}^{\pm}(|x|)$ have bounded derivatives under our
	assumption on $|x|$.
	\label{prop2.4}
\end{prop}
\begin{proof}If $\re\omega=0$, then we take $\xi\to\varepsilon\xi$ scaling to prove
	\eqref{onedajifen} since now the oscillatory integral is the Fourier transform of
	a $-1$ order Mihlin-type symbol function. We then assume from now on that
	$\re\omega\neq0$, which allows us to
        take the scaling $\xi\to\re\omega\xi$. Notice that  if
	$|\im\omega/\re\omega|\geq1$ or $|\xi|\notin(1/4,4)$ then we again have a
	$-1$ order Mihlin-type symbol function with symbol norms uniformly bounded, so we
	need only to consider
	 the following integral
	\begin{equation}
		|\re\omega|^{n-1}\int\frac{\beta(|\xi|)a(\re\omega|\xi|)e^{i\re\omega x\cdot\xi}}{|\xi|-1+i\varepsilon },\quad
		0<|\varepsilon |<1
		\label{reduced}
	\end{equation}
	in which $\beta(r)$ is a smooth function supported in $(1/4,4)$ and equal to
	$1$ in $(1/2,2)$. For simplicity we write $\beta(|\xi|)a(\re\omega|\xi|)$ as
	$\alpha(|\xi|)$. Now when $|\re\omega|\cdot|x|<1$, we can use the property that
	the
	function $\alpha(|\xi|)e^{i\re\omega x\cdot\xi}$ has bounded $\xi$ derivatives to
	show that the integral in \eqref{reduced} will be uniformly bounded under such assumption. 
	Therefore we need only to consider the case when $|\re\omega||x|\geq1$.

	Recall the following standard formula about the Fourier transform of
	sphere $S^{n-1}$:
	\begin{equation*}
		\int_{S^{n-1}}e^{x\cdot\omega}d\sigma(\omega)=\sum_{\pm}|x|^{-\frac{n-1}{2}}c_{\pm}(|x|)e^{\pm
		i|x|},
	\end{equation*}
	in which we have
	\begin{equation*}
		|\frac{d^j}{dr^j}c_{\pm}(r)|\geq r^{-j}, r\leq1/4.
	\end{equation*}
        So using polar coordinates and this formula we obtain the following integrals:
	\begin{equation}
		\sum_{\pm}
		(|\frac{\re\omega}{x}|)^{\frac{n-1}{2}}\int_{\frac{1}{4}}^{4}\frac{\alpha(r)c_{\pm}(r|\re\omega||x|)}{r-1+i\varepsilon
		}r^{\frac{n-1}{2}}e^{\pm i r \re\omega|x|}dr,\quad 0<|\varepsilon |<1.
		\label{integral}
	\end{equation}
	  Now let us 
	  recall the following lemma which was also proved in \cite{bssy}:
	\begin{lemma}
	 \begin{equation}
		 \int\frac{e^{-irt}}{r-1+i\varepsilon}dr=2\pi iH(\varepsilon t)e^{-it}e^{-|\varepsilon t|},
		 \label{f.t}
	 \end{equation}
	 in which $H(t)$ is the Heaviside function.
		\label{lemma2.1}
	\end{lemma}
	Now we can regard the integrals in \eqref{integral} as a convolution between the
	function in \eqref{f.t} and $b_{\pm}(t,|x|)$, 
	in which $b_{\pm}(t,|x|)$ denotes the $r$ Fourier transform of function
	$\alpha(r)c_{\pm}({r\re\omega|x|})r^{\frac{n-1}{2}}$. Notice that in particular we have
	\begin{equation}
		|D_x^{\gamma}D_t^{\alpha}b_{\pm}(t,|x|)|\leq C_{N,\gamma,\alpha}(1+t)^{-N}
		\label{b}
	\end{equation}
	due to the support of $|x|$ and the assumption that $|\re\omega|\cdot|x|>1$.
  	Without loss of generality we assume
	$\re\omega>0$ and $\varepsilon>0$ and the other cases can be handled similarly.
	Under such assumption the
	convolution is:
	\begin{equation}
		2\pi ie^{\mp i\re\omega|x|}(e^{-\re\omega\varepsilon|x|}\int_{-\infty}^{\re\omega|x|} e^{it}e^{\varepsilon
		t}b_{\pm}(t,|x|)dt).
		\label{convolution}
	\end{equation}
	Therefore the proof to the proposition will be completed as long as we can show
	that the function in the parentheses has uniformly bounded $x$ derivatives.

	In fact, we have
	\begin{equation*}\begin{split}
		&\partial _x^{\gamma}(e^{-\re\omega\varepsilon|x|}\int_{-\infty}^{\re\omega|x|} e^{it}e^{\varepsilon
		t}b_{\pm}(t,|x|)dt)\\
		=&\sum_{|j|+|k|=|\gamma|-1}C_{j,k}\partial_x^{j}(\re\omega\varepsilon|x|e^{-\re\omega\varepsilon|x|})\partial
		_x^{k}(e^{i\re\omega|x|}b_{\pm}(\re\omega|x|,|x|))\\&+\sum_{j+k=\gamma}C_{j,k}\partial
		_x^{j}(e^{i\re\omega\varepsilon|x|})\int_{-\infty}^{\re\omega|x|}e^{it}e^{-\varepsilon
		t}\partial
		_x^{k}(b_{\pm}(t,|x|))dt\\
		=&I+II.
	\end{split}
	\end{equation*}
	By the fact that $\varepsilon<1,|x|\approx 1$ and \eqref{b} we immediately see
	that $I$ is uniformly bounded. Now for $II$,
	after a $t\to -t+\re\omega$ variable substitution the terms in it are majorized  by the following 
	integral
	\begin{equation}\begin{split}
		({\re}\omega\varepsilon )^j\int_0^{\infty}e^{-\varepsilon
		t} |b_{\pm}^{(k)}(-t+{\re}\omega,|x| )|dt
		\leq C_{j,k,N}\int_0^{\infty}(\varepsilon {\re}\omega)^j(1+\varepsilon
		t)^{-j}(1+|t-{\re}\omega|)^{-N}dt. \end{split}
		\label{uniform_bound}
	\end{equation}
	By arguing the ratio between $t$ and $\re\omega$ we can immediately show that
	$II$ is also uniformly bounded. So the proof is complete.
	\end{proof}
Now, let $\varepsilon =\frac{2^j}{\lambda}$, so the kernel of $S_j$ is 
\begin{equation}
	S_j(x,y)=\frac{1}{\coe}\int_0^{\infty}\beta({t}/{\varepsilon
	})\rho(t)e^{\lambda t}e^{-\mu t}\cos tP(x,y)dt.\label{s_j}
\end{equation}
We should notice that due to the finite propagation speed of the wave operator $\cos tP$
the kernels $S_j(x,y)$ are actually supported in the region where $|\dis/\varepsilon|<2$,
so we have 
\begin{equation}
	S_j(x,y)=\rho(x,y,\varepsilon )S_j(x,y)
	\label{finite.speed}
\end{equation}
in which we recall that $\rho$ is a smooth bump function supported when $\dis\leq
4\varepsilon $. By such consideration we can therefore restrict the support of 
all the following operators to such a small region. 

By Euler's formula, in geodesic coordinates we have
\begin{equation}
	\cos tP(x,y)=\sum_{\pm}\int_{\R^n}e^{i\kappa(x,y)\cdot\xi}e^{i\pm
	t|\xi|}\alpha_{\pm}(t,x,y,|\xi|)d\xi +Q(t,x,y)\label{fourier.integral}
\end{equation}
in which $Q(t,x,y)$ is a smooth function with compact support in $t,x,y$,
$\kappa(x,y)$ are the geodesic coordinates of $x$ about $y$ and
$\alpha_{\pm}(t,x,y,|\xi|)$ are $0$-order symbol functions in $\xi$. So if we replace
the operator $\cos tP$ in \eqref{s_j} we have a operator with $(p,q)$ norm equal to
$O(\lambda^{-1}\varepsilon )$. Notice that due to the $t$ support in \eqref{s_j} we would
have that there are only $O(\log_2\lambda)$ many $S_j$ not vanishing. So summing over so
many $j$ will give us a $(p,q)$ bounded operator over the Sobolev line immediately. Notice
that by \eqref{finite.speed} we can assume that both the Fourier integral and $Q(t,x,y)$
are supported when $\dis/\varepsilon<4$.

So by abusing language a little bit, we can replace the wave operator $\cos tP$ with the
Fourier integral representation in \eqref{fourier.integral}. 
Now take a $(t,\xi)\to(\varepsilon t,\xi/\varepsilon )$ scaling in \eqref{s_j} we then
obtain
\begin{equation}
	S_j^{\pm}(x,y)=\frac{\varepsilon
	^{1-n}}{\coe}\int_{0}^{\infty}\int_{\R^n}\beta(t)\rho(\varepsilon
	t)e^{i\lambda\varepsilon t}e^{-\mu\varepsilon t}\alpha_{\pm}(\varepsilon
	t,x,y,{|\xi|}/{\varepsilon })e^{i\frac{\kappa(x,y)}{\varepsilon }\cdot\xi\pm
	t|\xi|}d\xi dt
	\label{k_j}
\end{equation}
for $\dis/\varepsilon <4$. However, when $\dis/\varepsilon \leq 1/4$, an integration by
parts argument with respect to $\xi$ would show that 
\begin{equation*}
	|S_j^{\pm}(x,y)|\leq \varepsilon ^{1-n}\lambda^{-1}\leq \dis^{2-n}2^{-j},
\end{equation*}
which is a $(p,q)$ bounded operator over the Sobolev line after summation over $j$. 
So we are reduced to considering the operator
$K_{j}^{\pm}=\beta(\frac{\dis}{2\varepsilon})S_j^{\pm}(x,y)$ in which $\beta(r)$ is supported when
$r\in(1/2,2)$.

We then proceed as in \cite{bssy}. More specifically, let $a_{\varepsilon
}^{\pm}(\tau,x,y,|\xi|)$ denote the inverse Fourier transform of 
\begin{equation*}
	t\to\beta(t)\rho(\varepsilon t)\alpha_{\pm}(\varepsilon t,x,y,|\xi|/\varepsilon ),	
\end{equation*}
in which
\begin{equation}
	|D_{\xi}^{\gamma}a_{\varepsilon }^{\pm}(\tau,x,y,|\xi|)|\leq
	C_{N,\gamma}(1+\tau)^{-N}|\xi|^{-|\tau|}.
	\label{tau.decay}
\end{equation}
Then after the Fourier transform in \eqref{k_j} we would have
\begin{equation}
	K_j^{\pm}(x,y)=\varepsilon ^{2-n}\int\int_{\R^n}e^{i\frac{\kappa(x,y)}{\varepsilon
	}\cdot\xi}\frac{a_{\varepsilon }^{\pm}(\tau,x,y,|\xi|)}{(\pm|\xi|-\tau-\varepsilon
	\lambda-i\varepsilon \mu)(\pm|\xi|-\tau+\varepsilon \lambda+i\varepsilon \mu)}d\xi
	d\tau.
	\label{dajifen}
\end{equation}
Now if we split the fraction in the integrand as following
\begin{equation*}\begin{split}
	&\frac{1}{(\pm|\xi|-\tau-\varepsilon
	\lambda-i\varepsilon \mu)(\pm|\xi|-\tau+\varepsilon \lambda+i\varepsilon
	\mu)}\\=&\frac{1}{2(\varepsilon \lambda+i\varepsilon
	\mu)}(\frac{1}{\pm|\xi|-\tau-\varepsilon \lambda-i\varepsilon
	\mu}-\frac{1}{\pm|\xi|-\tau+\varepsilon \lambda+i\varepsilon \mu})\end{split}
\end{equation*}
then we would have, after applying Proposition \ref{prop2.4},
\begin{equation}
	K_j^{\pm}(x,y)=2^{-j}\varepsilon ^{2-n}a_{1,\omega}(x,y)+2^{-j}\varepsilon ^{2-n}\sum_{\pm}\int e^{\pm i(\tau\pm\varepsilon
	\lambda)\frac{\dis}{\varepsilon }}\cdot|\tau\pm\varepsilon
	\lambda|^{\frac{n-1}{2}}a_{\varepsilon }(\tau,x,y)d\tau
	\label{split}
\end{equation}
in which $a_{1,\omega}(x,y)$ is a uniformly bounded smooth function with support when
$\dis/\varepsilon\in(1/4,4) $, and $a_{\varepsilon }(\tau,x,y)$ is virtually the function
$a_{2,\omega}$ in \eqref{onedajifen} and it particular, it has the
following properties
\begin{equation}
	|\partial_{x,y}^{\gamma}a_{\varepsilon }(\tau,x,y)|\leq
	C_{\gamma,N}(1+\tau)^{-N}\varepsilon ^{-|\gamma|}
	\label{c.s.pre.condition}
\end{equation}
also due to the fact that $\dis/\varepsilon \approx1$.

Now for the $2^{-j}\varepsilon ^{2-n}a_{1,\omega}(x,y)$ piece, a simple calculation shows that
its
$L^1\to L^{\infty}$ norm is $2^{-j}\varepsilon ^{2-n}$ and the $L^p\to L^p$ norm is
$2^{-j}\varepsilon ^2,1\leq p\leq\infty$. So interpolation shows that these operators sum up to a
$(p,q)$ bounded operator on the Sobolev line. Therefore it reduces to analyzing the second operator in
\eqref{split}. After carrying out the $\tau$ integral, we immediately see that we are
further
reduced, without loss of generality, to estimating the $(p,q)$ bounds over the Sobolev
line for the following operators
\begin{equation}
	T_j(x,y)=2^{-j}\varepsilon ^{2-n}2^{j\frac{n-1}{2}}e^{i\lambda\dis}a_{\varepsilon
	}(x,y)=\lambda^{\frac{n-3}{2}}\varepsilon ^{-\frac{n-1}{2}}e^{i\lambda\dis}a_{\varepsilon
	}(x,y)
	\label{c-s.operator}
\end{equation}
in which the smooth function $a_{\varepsilon}(x,y)$ is supported when
$\dis/\varepsilon\in(1/4,4)$, and 
\begin{equation}
	|\partial_{x,y}^{\gamma}a_{\varepsilon }(x,y)|\leq C_{\gamma}\varepsilon
	^{-|\gamma|}.\label{t_j.support}
\end{equation}

Now the rest part will be standard procedure as in \cite{Soggebook} since we know the function $d_g(x,y)$ satisfies the
Carleson-Sj\"{o}lin condition, and this  was also done similarly in \cite{Kenig}. See
\cite{KRS} also for the Euclidean case. For the reader's
convenience we state it briefly as follows. First, \eqref{c.s.pre.condition} reminds us that if we scale
$x,y$ back to unit-length by $(x,y)\to(\varepsilon x,\varepsilon y)$ we have
\begin{equation}
	|\partial_{x,y}^{\gamma}a_{\varepsilon }(\varepsilon x,\varepsilon y)|
	\leq C_{\gamma}.
	\label{scaled_derivative_estimate}
\end{equation}
Also the phase function now is $\lambda d_g(\varepsilon x,\varepsilon y)= 2^jd_g(\varepsilon
x,\varepsilon y)/\varepsilon $ and $d_g(\varepsilon x,\varepsilon y)/\varepsilon $ will
satisfy Carleson-Sj\"{o}lin condition such that the hypersurface
$\nabla_{x}d_g(\varepsilon x,\varepsilon y_0)/\varepsilon $ will have curvature bounded
away from zero. So, if $f(y)$ is a test function, then we have, if 
\begin{equation}
	\frac{1}{q}=\frac{n-1}{n+1}(1-\frac{1}{p}), \quad 1\leq p\leq 2,
	\label{cs-line}
\end{equation}
that
\begin{equation}
	\begin{split}
		&||T_j(x,y)f(y)||_{q}\\
		=&\lambda^{\frac{n-3}{2}}\varepsilon ^{-\frac{n-1}{2}}(\int|\int
		e^{-i\lambda d_g(x,y)}a_{\varepsilon }(x,y)f(y)dy|^pdx)^{\frac{1}{q}}\\
		=&\lambda^{\frac{n-3}{2}}\varepsilon ^{-\frac{n-1}{2}}\varepsilon
		^{n+\frac{n}{q}}(\int|\int
		e^{-i2^jd_g(\varepsilon x,\varepsilon y)/\varepsilon
		}a_{\varepsilon }(\varepsilon x,\varepsilon y)f(\varepsilon y)dy|^pdx)^{\frac{1}{q}}\\
		\lesssim&\lambda^{\frac{n-3}{2}}\varepsilon ^{-\frac{n-1}{2}}\varepsilon
		^{n+\frac{n}{q}}2^{\frac{-jn}{q}}\varepsilon ^{-\frac{n}{p}}||f||_p\\
		=&\lambda^{-2+n(\frac{1}{p}-\frac{1}{q})}2^{j(\frac{n+1}{2}-\frac{n}{p})}||f||_p.
	\end{split}
	\label{final_estimate}
\end{equation}
Here the inequality is due to the standard $n\times n$ Carleson-Sj\"{o}lin estimate in
\cite{Soggebook} (Theorem 2.2.1). On the other hand, if we apply Young's inequality to the kernel
$T_j(x,y)$ as in \eqref{c-s.operator}, we also see that the kernel on the line
$(1/p,0)$ will be $L^p\to L^q$ bounded by the same norm as in \eqref{final_estimate} with
corresponding exponent. Now
a simple interpolation shows that our local operator is $L^p\to L^q$ bounded when $(1/p,1/q)$
are on $\overline{CC'}$, with end points removed. 

\begin{remark}
	Through Lemma \ref{-2symbol} and the oscillatory
	integral \eqref{dajifen} the reader may have noticed that our local operator
	$\local(P)$ bears great
	similarity compared with the classical Hadamard parametrix to
	$(\Delta_g+\zeta)^{-1}$ used in \cite{Kenig}. This
	is in fact not surprising since they are essentially the same operator expressed
	in different ways. The reader may also compare this result with Theorem 2.2 in
	\cite{KRS}, which is the Euclidean case when we replace the local resolvent by the
	resolvent $(\Delta_{\R^n}+\zeta)^{-1}$ to see the similarity between the local
	operator and its Euclidean counterpart. 
	
	Also when $n=2$, due to the fact that a
	$-n$ order pesudodifferential operator in $\R^n$ has kernel bounded by
	$\log|x-y|^{-1}$ (see for example \cite{taylor}), the readers can easily check the
	local operator $\local(P)$ is now bounded from $H^1(M)$, the Hardy space on
	compact manifolds which is
	defined in \cite{strichartz}, to $L^{\infty}(M)$.
\end{remark}

\section{Non-local Operator}

Now let us deal with the non-local operator 
\begin{equation}
	r_{\lambda,\mu}(P)=\frac{1}{i(\lambda+i\mu)}\int_0^{\infty}(1-\rho(t))e^{i\lambda
	t}e^{-\mu t}\cos tPdt\label{nonlocal}
\end{equation}
in which we assume that $\mu\geq 1$. This operator is  more easier than its local
counterpart to handle and we prove:
\begin{theorem}
	The non-local operator $r_{\lambda,\mu}$(P) defined in \eqref{nonlocal} is a
	uniformly bounded operator from $L^p(M)$ to $L^q(
	M)$ if $1/p-1/q=2/n$ and 
	$p\leq {2(n+1)}/{n+3},q\geq {2(n+1)}/{(n-1)}$.
	Notice that this is the segment $\overline{AA'}$ given in figure \ref{figure}.
	\label{nonlocal_theorem}
\end{theorem}
Similar to \cite{bssy}, the proof is based on the following lemma:
\begin{lemma}
	Given a fixed compact Riemannian manifold of dimension $n\geq 3$ there is a
	constant $C$ so that whenever $\alpha\in C(\mathbb{R}_+)$ and let
	\begin{equation*}
		\alpha_k(P)f=\sum_{\lambda_j\in[k-1,k)}\alpha(\lambda_j)E_jf,k=1,2,\dots
	\end{equation*}
	then we have
	\begin{equation}
		||\alpha_k(P)f||_{L^q(M)}\leq C
		k(\sup_{\tau\in[k-1,k)}|\alpha(\tau)|)||f||_{L^p(M)}
		\label{lemma2.3}
	\end{equation}
	if $1/p-1/q=2/n$ and $p\leq2(n+1)/(n+3),q\geq2(n+1)/(n-1)$.
	\label{lemma2.3.1}
\end{lemma}
To prove the lemma firstly, let us recall the following theorem in \cite{Soggebook} and
\cite{Sogge1}:

\begin{theorem*}
	If $\chi_{\lambda}$ denotes the spectral cluster projection of the operator $P=\sqrt{-\Delta_g}$,
	namely $\chi_{\lambda}f=\sum_{\lambda_j\in[\lambda-1,\lambda)}E_jf$, then we have the following
	estimates
	\begin{equation}
		||\chi_{\lambda}f||_{L^2(M)}\leq
		C(1+\lambda)^{\delta(p)}||f||_{L^p(M)},\quad 1\leq
		p\leq\frac{2(n+1)}{n+3},
		\label{spectral1}
	\end{equation}
	and
	\begin{equation}
		||\chi_{\lambda}||_{L^q(M)}\leq
		C(1+\lambda)^{\delta(q)}||f||_{L^2(M)},\quad
		\frac{2(n+1)
		}{n-1}\leq q\leq \infty
		\label{spectral2}
	\end{equation}
	in which $\delta(p)=n|\frac{1}{p}-\frac{1}{2}|-\frac{1}{2}$.
	\label{sogge_theorem}
\end{theorem*}

Now since $\alpha_k(P)=\chi_k\alpha_k(P)\chi_k$, we have
\begin{equation*}
	\begin{split}
		||\alpha_k(P)f||_{L^q(M)}&\lesssim
		k^{\delta(q)}||\alpha_k(P)\chi_kf||_{L^2(M)}\\
		&\lesssim k^{\delta(q)}(\sup_{\tau\in[k-1,k)}|\alpha(\tau)|)||\chi_k
		f||_{L^2(M)}\\
		&\lesssim
		k^{\delta(q)+\delta(p)}(\sup_{\tau\in[k-1,k)}|\alpha(\tau)|)||f||_{L^p(M)}.
	\end{split}
\end{equation*}
On the Sobolev line we happen to have $\delta(p)+\delta(q)=1$. This completes the proof to the
lemma. Now to prove the Theorem \ref{nonlocal_theorem}, we just need to notice that under our
assumption $\mu\geq 1$ we have the following estimate:
\begin{equation}
	|r_{\lambda,\mu}(\tau)|\leq
	C_N\lambda^{-1}((1+|\lambda-\tau|)^{-N}+(1+|\lambda+\tau|)^{-N}),\quad
	\lambda,\mu\geq1.
	\label{lastone}
\end{equation}
So now the Lemma \eqref{lemma2.3.1} comes into application immediately if we take $N=3$ as
following:
\begin{equation}
	||r_{\lambda,\mu}(P)||_{L^p\to L^q}\leq
	C\sum_{k=1}^{\infty}k\lambda^{-1}(1+|\lambda-k|)^{-3}\leq C'
	\label{nonlocal.last}
\end{equation}

In the previous section we have already proved that 
that when $(1/p,1/q)$ are in the range $\overline{AA'}$ in figure \ref{figure} the local operator is
uniformly bounded between $L^p(M)$ and $L^q(M)$ with only the requirement that $\mu\neq0$. Theorem
\ref{main_theorem} is therefore proved.

\begin{proof}
	[Proof of Theorem \ref{equivalence}]
With the help of Lemma \eqref{lemma2.3.1} we can prove Theorem \ref{equivalence} now. 
As we have seen in Section 2 that if we replace the resolvent in \eqref{1.14} by
the local operator $\local(P)$ then the estimates hold automatically, therefore
we need only to prove that the non-local operator $r_{\lambda,\mu}(P)$ satisfies
the same estimates
\begin{equation}
	||r_{\lambda,\mu}f||_{L^q(M)}\leq C||f||_{L^p(M)},
	\label{nonlocal.p.q.}
\end{equation}
under the assumption of \eqref{1.13} and in particular when $|\mu|\lesssim1$.
This is again an application of
Lemma \ref{lemma2.3.1} with finer arguments.

First, integration by parts in
\eqref{nonlocal} shows that we have
\begin{equation}
	|r_{\lambda,\mu}(\tau)|\leq C\lambda^{-1}\left[
	(1+|\lambda-\tau|)^{-N}+|\mu|^{-1}(1+|\mu|^{-1}|\lambda-\tau|)^{-N} \right],\quad
	N=0,1,\dots
	\label{nonlocal.multiplier}
\end{equation}
Now, assume we have a positive number $\alpha$ such that
$\varepsilon_p(\lambda)\leq\alpha\leq1$ in which $p$ is any number within the range mentioned
in Theorem \ref{equivalence}. Notice that if we partition the interval
$[\lambda-\alpha,\lambda+\alpha]$ evenly into $O(\alpha/\varepsilon _p(\lambda))$ pieces
of small interval of length $\varepsilon _p(\lambda)$, then by Minkowski inequality, $L^2$
orthogonality and our special requirement on $\varepsilon _p(\lambda)$ that $\varepsilon
_p(\lambda)\approx\varepsilon _p(\lambda+1)$, we immediately have
\begin{equation}
	||\sum_{|\lambda-\lambda_j|\leq\alpha}E_jf||_{L^{p'}(M)}\leq
	C\lambda^{2\delta(p)}\alpha||f||_{L^p(M)}, \varepsilon_p(\lambda)\leq\alpha\leq1
	\label{small.to.large}
\end{equation}
In particular the constant $C$ is uniform. 
This amounts to saying that if the shrinking spectral estimates hold for a smaller spectral
cluster, they must hold as well for any larger cluster (but still shorter than unit
length, of course). Now using $L^2$ orthogonality again, if we have
$\lambda'\in[\lambda-1,\lambda+1]$, we have
\begin{equation}
	||\sum_{|\lambda'-\lambda_j|\leq\alpha}r_{\lambda,\mu}(\lambda_j)E_jf||_{L^{p'}(M)}\leq
	C\alpha\lambda^{2\delta(p)}(\sup_{|\lambda'-\tau|\leq\alpha}|r_{\lambda,\mu}(\tau)|)||f||_{L^p(M)}
	,\, \varepsilon
	_p(\lambda)\leq\alpha\leq1.
	\label{small.to.large.sup}
\end{equation}
So if we use $TT^*$ arguments to break
\eqref{small.to.large} into $L^p(M)\to L^2(M)$ and $L^2(M)\to L^q(M)$ then compose them
together while choosing $\alpha=|\mu|$, we  have
\begin{equation}
	||\sum_{|\lambda'-\lambda_j|\leq|\mu|}r_{\lambda,\mu}(\lambda_j)E_jf||_{L^{q}(M)}\leq
	C|\mu|\lambda(\sup_{|\lambda'-\tau|\leq|\mu|}|r_{\lambda,\mu}(\tau)|)||f||_{L^p(M)},\, 
	\label{mu.p.q}
\end{equation}
in which $|\mu|\geq\max\left\{ \varepsilon _p(\lambda),\varepsilon _{q'}(\lambda)
\right\}$. 

Now for the non-local operator we have
\begin{equation}
	||r_{\lambda,\mu}f||_{L^q(M)}\leq||\sum_{|\lambda-\lambda_j|>1}r_{\lambda,\mu}(\lambda_j)E_jf||_{L^q(M)}
	+||\sum_{|\lambda-\lambda_j|\leq1}r_{\lambda,\mu}(\lambda_j)E_jf||_{L^q(M)}.
	\label{nonlocal.fenduan}
\end{equation}
For the first summand, we simply use the Lemma \ref{lemma2.3.1} to control it. For the
second summand, we can evenly partition $[\lambda-1,\lambda+1]$ into small intervals
$I_k$ of
length comparable to $|\mu|$ as:
\begin{equation}
	I_k=\left\{ \tau\in[\lambda-1,\lambda+1]:k|\mu|\leq|\tau-\lambda|\leq(k+1)|\mu|
	\right\},\, k=0,1,\dots
	\label{ik}
\end{equation}
then take sum using \eqref{nonlocal.multiplier} and
\eqref{mu.p.q} to see
\begin{equation}\begin{split}
	&||\sum_{|\lambda-\lambda_j|\leq1}r_{\lambda,\mu}(\lambda_j)E_jf||_{L^q(M)}\\
	\leq&C\lambda^{-1}\lambda|\mu|\sum_{k}\left[ (1+|\mu| k)^{-N}+|\mu|^{-1}(1+k)^{-N}
	\right]||f||_{L^p(M)}\\
	\leq&C||f||_{L^p(M)}.
\end{split}
	\label{equivalence.last}
\end{equation}

\end{proof}
\begin{remark}
	When points $(1/p,1/q)$ are off the range $\overline{AA'}$ in figure
	\ref{figure}, the above technique we used to show the boundedness of the non-local
	operator will not work. In fact, when
	$2(n+1)/(n+3)<p\leq 2$ we have the following spectral projection
	estimates
	\begin{equation}
		||\chi_{\lambda}f||_{L^2(M)}\leq
		C(1+\lambda)^{\frac{(n-1)(2-p)}{4p}}||f||_{L^p(M)}.
		\label{strange.estimate}
	\end{equation}
	Now assume that we have some $(1/p,1/q)$ being on the Sobolev line but on the left hand
	side of point $A$, then after the $L^p\to L^2$ and $L^2\to L^q$ composition we  have
	\begin{equation}
		||\chi_{\lambda}f||_{L^q(M)}\leq
		C(1+\lambda)^{\frac{(n-1)(2-p)}{4p}+\delta(1/(\frac{1}{p}-\frac{2}{n}))}||f||_{L^p(M)}\label{p.q.estimate}
	\end{equation}
	and an easy calculation  shows that now the power of $\lambda$ is within
	$(1,3/2]$, which is not sufficient to control the non-local operator as the readers
	have seen. However, unlike the $L^p\to L^{p'}$ estimate, a general $L^{p}\to
	L^{q}$ estimate for
	$p$ and $q$ being off the line of duality obtained by composing two projections
	together may not be sharp. Therefore further improvement on the range
	$\overline{AA'}$ may still be possible.
	In fact, if we assume $M=S^3$ and $\mu=1$, then at the
	end point $p=4/3$ and $q=12$ the resolvent estimates now read as
	\begin{equation}
		||f||_{L^{12}(S^3)}\leq C
		||(-\Delta_{S^3}+(\lambda+i)^2)f||_{L^{\frac{4}{3}}(S^3)}.
		\label{s3}
	\end{equation}
	If we let $f$ be an arbitrary $L^2$ normalized eigenfunction $e_{\lambda}$ corresponding to
	eigenvalue $\lambda$, then we ought to have
	\begin{equation}
		||e_{\lambda}||_{L^{12}(S^3)}\lesssim\lambda||e_{\lambda}||_{L^{\frac{4}{3}}(S^3)}.
		\label{eigen.3}
	\end{equation}
	A further testing with spherical harmonics and zonal functions indicate that the
	inequalities above are not saturated, which casts some doubt on the sharpness of the
	admissible range for $(p,q)$. But right now we are not able to prove, nor disprove
	the sharpness of that range even for round spheres. As Sogge pointed out to us 
	such difficulties were encountered during his study on Bochner-Riesz means
	in \cite{Soggethesis}, \cite{sogge2}, which may indicate that to prove or disprove
	the sharpness of the range $\overline{AA'}$ in our problem might be substantially difficult.
	
	\label{riesz.means}	
\end{remark}

\section{Non-positive curvature manifolds}

Recall that in \cite{bssy} we proved the following result
\begin{equation}
	||\sum_{|\lambda_j-\lambda|<1/\log\lambda} E_jf||_{L^{\frac{2n}{n-2}}(M)}\leq
	C(\log\lambda)^{-1}\lambda||f||_{L^{\frac{2n}{n+2}}(M)},\quad\lambda\gg1,
	\label{old.neg.result}
\end{equation}
which is a special case of a recent unpublished estimate of Hassell and Tacey, 
also is an $L^p$ variant of earlier supernorm bounds implicit of B\'{e}rard \cite{berard}.
Now we shall prove the related result:
\begin{theorem}
	If $(M,g)$ is a compact manifold of dimension $n\geq3$ with non-positive
	sectional curvature then we have, when $p<\frac{2(n+1)}{n+3}$
	\begin{equation}
		||\sum_{|\lambda_j-\lambda|<1/\log\lambda} E_jf||_{L^2{(M)}}\leq
		C(\log\lambda)^{-\frac{1}{2}}\lambda^{\delta(p)}||f||_{L^p{(M)}},\quad\lambda\gg1.
		\label{new.neg.result}
	\end{equation}
	\label{neg.theorem}
\end{theorem}
By Theorem \ref{equivalence} this will immediately proves Theorem
\ref{neg.result}. 

Before we go through the details of the proof, we want to point out that if we can prove
\eqref{l2.neg} with $\log\lambda$ replaced by $\varepsilon \log\lambda$ in which
$\varepsilon $ is smaller than 1 and fixed, and only depends on $M$ and $p$, then $L^2$
orthogonality would immediately show that \eqref{new.neg.result} is proved with a larger constant $C$
possibly depending on $p$. But certainly this is harmless for us.

So we need to prove \eqref{new.neg.result} with $\varepsilon \log\lambda$, in which the
specific value of $\varepsilon $ is about to be determined later. First, 
we claim that if we choose an even nonnegative function
$a\in\mathcal{S}(\R)$ satisfying $a(r)=1, |r|\leq 1/2$ and having its Fourier transform supported in
$(-1,1)$, then in order to \eqref{new.neg.result} it is sufficient to prove that for the
multiplier
\begin{equation}
	a(\varepsilon \log\lambda(\lambda-P))=\frac{1}{2\pi\varepsilon \log\lambda}\int \hat
	a(t/(\varepsilon \log\lambda))e^{it\lambda}e^{-itP}dt
	\label{neg.multipler}
\end{equation}
we have\begin{equation}
	||a(\varepsilon \log\lambda(\lambda-P))f||_{L^{p'}(M)}\leq
	(\varepsilon \log\lambda)^{-1}\lambda^{2\delta(p)}||f||_{L^p(M)},\quad p<\frac{2(n+1)}{n+3}.
	\label{new.multiplier.estimate}
\end{equation}
In fact, due to the non-negativity of $a(r)$ especially $a(r)\approx1$ when $r$ is near $0$, we know that
if we use the fact $a=(\sqrt{a})^2$
then a $TT^*$ argument and the above estimate will
immediately imply that 
\begin{equation}
		||\sum_{|\lambda_j-\lambda|<1/(\varepsilon \log\lambda)} E_jf||_{L^2{(M)}}\leq
		C(\varepsilon \log\lambda)^{-\frac{1}{2}}\lambda^{\delta(p)}||f||_{L^p{(M)}}\label{l2.neg}
\end{equation}
due to $L^2$ orthogonality.

We then proceed in the way that we proved the estimates on the local operator
$\local(P)$, say breaking the $t$ interval into one part  
when $t\leq 1$ and the other one when $1\leq t\lesssim \log\lambda$ (c.f. \eqref{S_j} and
\eqref{S_0}).


Now, if $\psi\in C^{\infty}(\R^1)$ is an even function
and $\psi(r)=1$ when $|r|>2$ and $\psi(r)=0$ then we claim
that the operator defined as
\begin{equation*}
	b_{\lambda}(P)=\frac{1}{2\pi\loga}\int
	(1-\psi(t))\hat{a}({t}/({\loga}))e^{i(\lambda-P)t}dt
\end{equation*}
satisfies the estimate in \eqref{new.multiplier.estimate} when $p\leq 2(n+1)/(n+3)$.
This can be proved very easily by the $L^p\to L^{p'}$ version of Lemma \ref{lemma2.3} and
the fact that 
\begin{equation*}
	|b_{\lambda}(\tau)|\leq C{(\loga)}^{-1}(1+|\lambda-\tau|)^{-N}.
\end{equation*}
So we need only to consider 
\begin{equation}
	\frac{1}{2\pi\loga}\int(1-\psi(t))\hat{a}(t/(\loga))e^{i(\lambda-P)t}dt.
\end{equation}
To proceed, we need to replace the $e^{-itP}$ in the integrand above by
$\cos{tP}$ since we are going to use the latter's Huygens principle. Now, notice that since both $\psi$ and $a$ are even functions,
the difference between the operator in the above formula and
\begin{equation*}
	\frac{1}{2\pi\loga}\int\psi(t)\hat{a}(t/(\loga))e^{i\lambda t}\cos tPdt.
\end{equation*} which
is a smoothing operator with size of $O(\lambda^{-N})$, as $P$ is a positive
operator. So we are reduced to prove that
if we let $a_{\lambda}(P)$ denote
\begin{equation*}
	a_{\lambda}(P)=\int\nega e^{i\lambda t}\cos tPdt
\end{equation*}
then
\begin{equation}
	||a_{\lambda}(P)f||_{L^{p'}(M)}\leq C\lambda^{2\delta(p)}||f||_{L^p(M)},\quad
	p<2(n+1)/(n+3).
	\label{need.to.prove}
\end{equation}
We are going to use interpolation to prove \eqref{need.to.prove}. More specifically we
want to prove that
\begin{equation}
	||a_{\lambda}(P)||_{L^1\to L^{\infty}}\leq C\log\lambda e^{c\loga}\lambda^{(n-1)/2}
	\label{1.to.infty}
\end{equation}
in which $c$ is a small number depending on the geometry of the manifold $M$
and
\begin{equation}
	||a_{\lambda}(P)||_{L^2\to L^2}\leq C\log\lambda.
	\label{2.to.2}
\end{equation}
The second one is obvious since we have the fact that $\cos tP$ is a bounded $L^2$
operator. The first one was already proved in \cite{bssy} by using the fact that
\begin{equation*}
	\cos t\sqrt{-\Delta_g}(x,y)=\sum_{\gamma\in\Gamma}\cos t\sqrt{-\Delta_{\tilde g}}(x,\gamma y),x,y\in D
\end{equation*}
in which $D\subset \R^n$ is a fundamental domain of the universal covering map $\Pi:\R^n\to
M$, $\Gamma$ is the fundamental group, or the deck transform group of $M$, and $\tilde g$
is the pull-back of Riemannian metric by $\Pi$. 

Now we just need to do the interpolation. By the fact that
$\log\lambda\in o(\lambda^{\varepsilon })$ for an arbitrarily small $\varepsilon >0$,
after interpolation we have when $p<2(n+1)/(n+3)$ that there is a number 
\begin{equation*}
	\varepsilon (p)=(n+1)(\frac{1}{p}-\frac{n+3}{2(n+1)})>0
\end{equation*}
so that 
\begin{equation*}
	||a_{\lambda}(P)||_{L^{p'}(M)}\leq C\lambda^{2\delta(p)}
	\lambda^{\frac{2-p}{p}c\varepsilon _1-\varepsilon (p)}||f||_{L^p(M)}\leq
	C\lambda^{2\delta(p)}||f||_{L^p(M)}
\end{equation*}
if we choose $\varepsilon _1$ small enough according to $\varepsilon (p)$ and $c$. So
\eqref{need.to.prove} is proved.  Notice that due to the appearance of $\log\lambda$ we
are not able to prove the end point estimate when $p=2(n+1)/(n+3)$. In fact, by using
Lemma \ref{lemma2.3.1} we showed in \cite{bssy} that when $p=2(n+1)/(n+3)$ we can only
obtain
the following bound
\begin{equation*}
	||a_{\lambda}f||_{{L^{\frac{2(n+1)}{n-3}}}(M)}\leq
	C\lambda^{\frac{n-1}{n+1}}\log\lambda||f||_{L^{\frac{2(n+1)}{n+1}}(M)}
\end{equation*}
which is $\log\lambda$ worse than the \eqref{need.to.prove}.

\section{Torus $\torus$}

In this section we are going to prove Theorem \ref{torus.theorem} in a similar way
to the non-positive curvature manifold case in the previous section.
In fact, by an argument similar to the one prior to
\eqref{need.to.prove}, we need only to study that if we define an operator for
$0<\varepsilon (p)\leq1$ as the following
\begin{equation}
	a_{\lambda}(P)=\int\psi(t)\hat{a}(t/\lambda^{\varepsilon (p)})e^{it\lambda}\cos
	tPdt,
	\label{torus.multipler.expression}
\end{equation}
then we can have
\begin{equation}
	||a_{\lambda}(P)f||_{L^{p'}(\torus)}\leq \lambda^{2\delta(p)}||f||_{L^p(\torus)},
	p<\frac{2(n+1)}{n+3}.
	\label{need.to.prove.torus}
\end{equation}
Here as before $\psi(t)$ is a smooth function with support outside $(-4,4)$ and equals to one
$|t|>10$, and $\hat{a}(t)$ is supported in $(-1,1)$. Both functions are even, similar to the
non-positive curvature manifold case. 

As we have seen in the proof to Theorem \ref{equivalence}, if we can find such a value of
$\varepsilon (p)$ which satisfies \eqref{need.to.prove.torus}, then any smaller positive
number than $\varepsilon (p)$ will do as well. Therefore
in the following argument, we can simply focus on finding a largest possible $\varepsilon (p)$
based on an interpolation argument similar to the one used in previous section.

Now we are going to prove the following estimates to interpolate with:
\begin{equation}
	||a_{\lambda}f||_{L^{\frac{2(n+1)}{n-1}}(\torus)}\leq C
	\lambda^{\frac{n-1}{n+1}}\lambda^{\varepsilon (p)}||f||_{L^{\frac{2(n+1)}{n+3}}(\torus)},
	\label{torus.interpolation.1}
\end{equation}
and
\begin{equation}
	||a_{\lambda}f||_{L^{\infty}(\torus)}\leq C\lambda^{\frac{n-1}{2}}\lambda^{\varepsilon
	(p)\frac{n+1}{2}}||f||_{L^{1}(\torus)}.
	\label{torus.interpolation.2}
\end{equation}
The first one is an easy application of Lemma \ref{lemma2.3.1} so we need only to prove
the second one. In \cite{bssy} we presented the proof in $n=3$ case and the general
dimensional case is similar. For the readers' convenience we shall sketch it now. 

Recall that if we identify $\torus$ with its fundamental domain
$Q=(-\frac{1}{2},\frac{1}{2}]^n$ in $\R^n$, then we have
\begin{equation}\label{torus.half.wave}
	\cos t\sqrt{-\Delta_{\torus}}(x,y)=\sum_{j\in \Z^n} \cos t\sqrt{-\Delta_{\R^n}}(x-y+j),\quad x,y\in Q
\end{equation}
then by the finite speed of propagation of $\cos tP$ we are reduced in estimating the size
of the following integral
\begin{equation}\begin{split}
	A_1(x,y)&=
\sum_{\substack{j\in\Z^n\\|j|\leq\lambda^{\varepsilon
	(p)}}}\int\int
	e^{i(x-y+j)\cdot\xi}\psi(t)\hat{a}(t/\lambda^{\varepsilon (p)})e^{it\lambda}\cos
	t|\xi|dtd\xi\\
	&=
	\sum_{\substack{j\in\Z^n\\|j|\leq\lambda^{\varepsilon
	(p)} \\ |x-y+j|\geq1}}\int\int
	e^{i(x-y+j)\cdot\xi}\psi(t)\hat{a}(t/\lambda^{\varepsilon (p)})e^{it\lambda}\cos
	t|\xi|dtd\xi\\
	&\hspace{1cm}
	+\sum_{\substack{j\in\Z^n\\|j|\leq\lambda^{\varepsilon
	(p)} \\ |x-y+j|<1}}\int\int
	e^{i(x-y+j)\cdot\xi}\psi(t)\hat{a}(t/\lambda^{\varepsilon (p)})e^{it\lambda}\cos
	t|\xi|dtd\xi\\
	&=(I)+(II).
	\end{split}
	\label{A_1}
\end{equation}
Notice that $(II)$ will disappear in odd dimension due to Huygens principle. Nonetheless
it does not cause any harm in even dimensions neither due to the following simple argument.
In fact, due to our choice on the fundamental domain $Q$, there are only $O(2^n)$ many
terms non-vanishing in $(II)$, so by Euler's formula we need only to prove that 
\begin{equation}
	\int\int_{\R^n}e^{i(x-y+j)\cdot\xi}\psi(t)\hat{a}(t/\lambda^{\varepsilon
	(p)})e^{it(\lambda\pm|\xi|)}d\xi dt\in O(\lambda^{-N}),\quad |x-y+j|<1.
	\label{II}
\end{equation}
Integrating by parts with respect to $t$ shows that we
need only  prove \eqref{II} when an extra cut-off function $\beta({|\xi|/\lambda})$ is
inserted in the
integrand for function $\beta$ is as in . Then due to the fact that in the support of integrand we have $|t|>4$, another
integration by parts in $\xi$ variable completes the proof.

Now we notice that in $(I)$ if we replace $\psi(t)$ by $1-\psi(t)$ we will end up
with $O(2^n)$ many integrals like
\begin{equation*}
	\int_{\R^n}e^{i(x-y+j)\cdot\xi}\left(\Psi(\lambda-|\xi|)+\Psi(\lambda+|\xi|)\right)d\xi,\quad
	|x|>1,
\end{equation*}
in which $\Psi(r)$ is a Schwartz function. Now using polar coordinates $\xi\to r\omega$
will immediately show that these integrals are in $O(\lambda^{\frac{n-1}{2}})$, which
is better than \eqref{torus.interpolation.2}. So we need only to prove
\begin{equation}
	\sum_{\substack{j\in\Z^n\\|j|\leq\lambda^{\varepsilon
	(p)}\\|x-y+j|>1}}|\int_{\R^n}e^{i(x-y+j)\cdot\xi}\lambda^{\varepsilon
	(p)}a(\lambda^{\varepsilon (p)}(\lambda\pm|\xi|))d\xi|\leq
	\lambda^{\frac{n-1}{2}+\varepsilon (p)\frac{n+1}{2}}.
	\label{final}
\end{equation}
This can be proved immediately if again we use polar coordinates and the decay estimate of
the Fourier transform of the spheres.

After proving \eqref{torus.interpolation.1} and \eqref{torus.interpolation.2} we can now
do the interpolation. Let $0\leq t\leq1$ be determined by the following equation 
\begin{equation}
	t+(1-t)\frac{n+3}{2(n+1)}=\frac{1}{p}, p\leq \frac{2(n+1)}{n+3},
	\label{t.equation}
\end{equation} 
then the interpolation shows that we have
\begin{equation}
	||a_{\lambda}f||_{L^{p'}(\torus)}\leq
	\lambda^{\frac{n-1}{n+1}(1-t)+\varepsilon(p)(1-t)+\frac{n-1}{2}t+\varepsilon
	(p)\frac{n+1}{2}t}||f||_{L^p(\torus)}, p\leq\frac{2(n+1)}{n+3}.
	\label{interpolation.a.lambda}
\end{equation}
So we need only to solve for a positive $\varepsilon (p)$ as the largest possible value which
satisfies \eqref{need.to.prove.torus} when $t>0$ from the following
equation
\begin{equation}
	\varepsilon (p)=\frac{
	\frac{n(n-1)t+n-1}{2(n+1)}-\frac{n-1}{n+1}(1-t)-\frac{n-1}{2}t}{\frac{n+1}{2}t+1-t},
	\, t=\frac{2(n+1)}{n-1}\frac{1}{p}-\frac{n+3}{n-1}.
	\label{epsilon.p}
\end{equation}
An elementary derivative test shows that this function is increasing when $t\in[0,1]$ and
$\varepsilon (2(n+1)/(n+3))=0$, which coincides with our interpolation as the
$(1,\infty)$ endpoint
has a better bound. In particular when $p=2n/(n+2)$ we have $\varepsilon
(p)=1/(n+1)$. 

Now what we need is simply to compose the projections between $L^p\to L^2$ and $L^2\to
L^q$. Here we just need to choose the weaker estimates during the composition, say for
a general $(1/p,1/q)$ pair in the $\overline{AA'}$ admissible range we just choose
$\varepsilon (p)$ when $(1/p,1/q)$ is below the line of duality, and $\varepsilon (q')$ when
it is above it. So not surprisingly the improvement is symmetric with respect to the line
of duality. The closer the exponent $(1/p,1/q)$ is to the middle point $F$ in figure \ref{figure}, 
the better the improvement we can have.



\begin{thebibliography}{MA}
\bibitem{bssy} J. Bourgain, P. Shao, C. D. Sogge and X. Yao: {\em On $L^p$ resolvent
		estimates and the density of eigenvalues for compact Riemannian
		manifolds}, arXiv:1204.3927, submitted.
\bibitem{berard}  P. H. B\'erard: {\em On the wave equation on a compact
manifold without conjugate points}, Math. Z. {\bf 155} (1977), 249--276.

\bibitem{KRS}
C. E. Kenig, A. Ruiz and C. D. Sogge: {\em Uniform Sobolev inequalities and unique continuation for
second order constant coefficient differential operators}, Duke Math. J. {\bf 55} (1987),
329¨C347.
\bibitem{Kenig} D. Dos Santos Ferreira, C. Kenig and M. Salo:
        {\em On $L^p$ resolvent estimates for Laplace-Beltrami operators on compact manifolds},
        arXiv:1112.3216, and to appear in Forum Math.
\bibitem{SS} A. Seeger and C. D. Sogge: {\em On the boundedness of functions of (pseudo-) 
	differential operators on compact manifolds}. Duke Math. J. {\bf 59} (1989), 709--736.
\bibitem{Shen} Z. Shen: {\em On absolute continuity of the periodic Schr\"odinger operators}, 
	Internat. Math. Res. Notices {\bf 1} (2001), 1--31.
\bibitem{Soggethesis}C. D. Sogge: {\em Oscillatory integrals and spherical harmonics}, 
	Duke Math. J. {\bf 53} (1986), 43--65.
\bibitem{Sogge1} C. D. Sogge: {\em Concerning the $L^p$ norm of spectral clusters for 
	second-order elliptic operators on compact manifolds}, J. Funct. Anal. {\bf 77} (1988), 123--138.
\bibitem{sogge2} C. D. Sogge: {\em On the convergence of Riesz means on compact manifolds},
	Ann. Math. {\bf 126} (1987), 439--447
\bibitem{Soggebook} C. D. Sogge: {\em Fourier integrals in classical analysis,} Cambridge University Press, 1993.
\bibitem{Stein} E. M. Stein: {\em Harmonics Analysis: real-variable methods, orthogonality
	, and oscillatory integrals}, Princeton University Press, Princeton, NJ, 1993. 
\bibitem{strichartz}R. S. Strichartz: 
	{\em The Hardy space $H^1$ on manifolds and submanifolds}. Canad. J. Math. {\bf
	24} (1972),
	915--925. 
\bibitem{taylor} M. Taylor {\em Pseudodifferential operators and nonlinear PDE}.
	Progress in Mathematics, 100. Birkh\"{a}user Boston, Inc., Boston, MA, 1991.

\end{thebibliography}
\end{document}